\documentclass[11pt]{article}

\usepackage{amsmath}
\usepackage{amsthm,amscd,amsfonts}
\usepackage{amssymb, upref, color}
\usepackage{amsmath,amssymb,amsthm,mathrsfs}
\usepackage{color}
\usepackage[dvips]{graphicx}
\usepackage{epsf,epsfig, verbatim} %subfigure,
\usepackage{latexsym,bm}
\usepackage{indentfirst}
\usepackage{color}
\usepackage{cite}

\usepackage{epsfig}
\usepackage{epstopdf}
\usepackage{subfigure}
\usepackage[font=scriptsize]{caption}
\usepackage{graphicx}
\graphicspath{{figures/}}

\numberwithin{equation}{section}

\theoremstyle{plain}
\newtheorem{exam}{Example}[section]
\newtheorem{theorem}[exam]{Theorem}

\newtheorem{definition}[exam]{Definition}

\begin{document}
	\sloppy
	\captionsetup[figure]{labelfont={bf},name={Fig.},labelsep=period}
	\captionsetup[table]{labelfont={bf},name={Table},labelsep=period}

	\title{Variational formulation for stratified steady water wave in   two-layer flows
		  \footnote{ Y. X. was supported by the Natural Science Foundation of Zhejiang Province (No.LZ24A010006), and NNSFC (No. 11931016).  Z. Z. was supported by the National Natural Science Foundation of
China (Grant Nos. 12271509, 12090010 and 12090014).}	}
	\author
	{
		Yuchao He$^{a}$\,\,\,\,\,Yonghui Xia$^{a}\footnote{Corresponding author.}$\,\,\,\,\,Zhe   Zhou$^{b,c}$		
		\\
		{\small \textit{$^a$ School of Mathematical  Science,  Zhejiang Normal University, 321004, Jinhua, China}}\\
		{\small Email:  YuchaoHe@zjnu.edu.cn; xiadoc@163.com; yhxia@zjnu.cn}
		\\
		{\small \textit{$^b$ Academy of Mathematics and Systems Science,
Chinese Academy of Sciences,
Beijing 100190,
China.}}\\
{\small \textit{$^c$ School of Mathematical Sciences, University of Chinese Academy of Sciences, 100049,  Beijing,   China}}\\
		{\small Email : zzhou@amss.ac.cn}
	}
	\maketitle
	\begin{abstract}
		In this paper, the  variational formulation for steady periodic stratified water waves in two-layer flows is given. The critical
		points of a natural energy functional is proved to be the solutions of the governing equations. And the second variation of the functional is also presented.\\
		
		
		\end{abstract}

	\section{Introduction}
	Since the beginning of this century, the stratified water wave model has attracted widespread attention and numerous scholars to study it. The stratified water wave rendered numerous insight to attack the general problems. For example, in equatorial water waves, the equatorial wave-current interactions and the temperature changes contribute to the emergence of the pronounced density stratification  \cite{cons-i,cons-j,2layer}.
	Moreover, people   frequently observed that water columns  have nearly constant density outside of thin transition layers called pycnoclines. As the water wave passes through a pycnocline, the density experiences something close to a jump discontinuity  \cite{walsh2}. Indeed, a common and effective method  is to  view these waves as being a layering of multiple immiscible fluids each with its own constant  density. And the dynamic behavior of such stratified water wave is characterized by  multiple-layer Euler governing equations with a piece-wise constant density function.  %In this model, the density distribution is piece-wise analytic.
	% An effective method to characterize the dynamic behavior of such stratified water wave is to use a model with a piece-wise constant density function.
	 There are already quite a few insightful works of piece-wise constant density water waves, including but not limited to \cite{2layer,c-w,c-i, walsh3, wheeler}.
	
	However, from a physical perspective, the piece-wise constant density water waves are still very special. Due to the complexity of  salinity and temperature gradient, the actual stratification situation is more complex, here we  refer the reader to \cite{c-c,jackson} for numerous significant works in fluid mechanics and oceanography. A more nature  approach is to consider continuous stratification.
	The continuous stratified water wave problem was originally initiated
	by Dubreil-Jacotin \cite{zuizao} in 1934. Since then, many scholars have conducted research on the issue of water wave with continuous
	stratification. In 2009, Walsh \cite{walsh1} extended  partial results of Constantin and Strauss \cite{cons-strauss}   to the problem of variable density.  And the existence of global continuous  solutions was proved. In 2011, Escher et al. \cite{escher}  used bifurcation theory to construct small periodic gravity  water waves with continuous stratification. Moreover, Henry and Matioc \cite{henry1}  studied the   capillary-gravity water
	waves with continuous stratification and gave the global bifurcation result.
	After the work of \cite{henry1}, Henry and Matico presented the existence of steady periodic capillary-gravity
	 water waves with  stratification. In 2021, Haziot \cite{haziot} proved that the solutions of large-amplitude steady stratified periodic water waves with the variable density  exist.  In 2023, Xu et al. \cite{X-L-Z} studied the symmetry of continuous stratified water waves.
	
	 Recently, Chen and Walsh\cite{walsh2} presented a new layered model wherein $\Omega=\mathop{\cup}\limits_{i=1}^N\Omega_i$ is partitioned into finitely many immiscible fluid
	regions. In each fluid region $\Omega_i$, the density  distribution is analytic. However, the density varies sharply at the junction of different regions. And Chen and Walsh \cite{walsh2} innovatively proposed the method to recover the wave form some date in the ocean bed.
	In this paper, we  conduct research on Chen and Walsh's \cite{walsh2}   model and present the variational formulation for  two-layer stratified
	water flows and the linear stability results.
Considering the convenience of readers for reading, we only consider the case $N=2$.
	Detailed formulations and dynamic boundary conditions will be provided in the next section.
	For rotational steady water waves, variational formulation was presented by Constantin et al. \cite{cons1}. And based on the variational formulations presented in \cite{cons1}, the linear stability properties and the formal stability properties of rotational steady water waves were studied by Constantin and Strauss \cite{cons2}.
	 It is worth noting that, Chu et al. \cite{2layer} made laying
	a foundation contribution to the works of the variational formulation  in two-layer water waves, in which   the density  is piece-wise constant function. Inspired by the method of Constantin et al. \cite{cons1}, Constantin and Strauss\cite{cons2}, Chu  et al. \cite{2layer},  we  construct a new variational formulation for two-layer water waves with continuous stratification. Different from the work of  \cite{2layer}, the density function in this paper is only require to be piece-wise analytical, which is more  complicated but more  natural from a physical perspective.

	The plan of this paper is as follows: in Section 2,
	we state a two-layer stratified water wave model and present the Euler equations and equivalent formulations. In Section 3, we present variational formulation and show that critical
	points of a natural energy functional are solutions to the governing equations derived from first principles.
Finally, in Section 4, we present the linear stability results by calculating
the second variation of the functional.

	\section{Formulation of stratified water wave }
	\subsection{Euler equations }
	We formulate the problem in a Cartesian coordinate system $(x, y)$ in $\mathbb R^2$, where $x$ is the horizontal and
	$y$ the vertical direction, respectively. The fluid region is defined as $\Omega=\Omega_1\cup\Omega_2$ and
	\[
	\begin{split}
	\Omega_1&=\{(x,y)\in\mathbb R^2:-\pi<x<\pi,\ -d<y<\widetilde\eta(x)\},\\
	\Omega_2&=\{(x,y)\in\mathbb R^2:-\pi<x<\pi,\ \widetilde\eta(x)<y<\eta(x)\},
	\end{split}
	\]
where $d > 0$ represents the depth of the rigid bottom,  $\eta(x, t)$ is the free surface of the upper layer fluid, $\widetilde{\eta}(x)$ represents the free surface of the internal wave.
$(u_i(x,y),v_i(x,y))$ represent the velocity field in $\Omega_i$, $i=1,2$ (see Fig. 1).
\begin{figure}[h]
	\centering
	\includegraphics[scale=0.8]{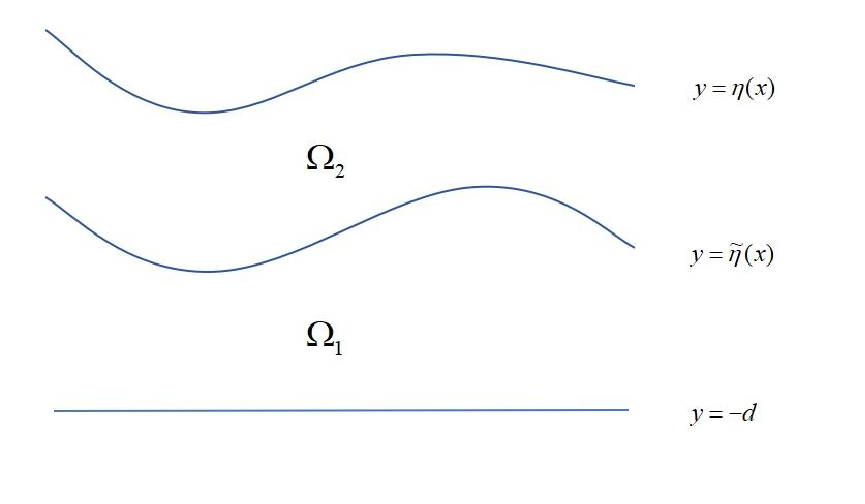}
	\caption{Two-layer stratified water wave  consists of  $\Omega_1$ and $\Omega_2$.  The internal boundary $\widetilde{\eta}(x)$ represents the pycnoclines. As the water wave passes through $\widetilde{\eta}(x)$ from $\Omega_1$ to $\Omega_2$, the density  jumps discontinuity from $\rho_1$ to $\rho_2$.}
	\label{figure}
\end{figure}

 Let $\rho_i$ be the density  of the flow in $\Omega_i$, for $i=1,2$. It should be noted that $\rho_i$ are not  constants. In this paper, we assume that the density increases with settlement depth, which implies $\rho_1>\rho_2$. Inspired by Walsh \cite{walsh1}, we present
the  Euler equations of two-layer water wave together with conservation of mass conditions and  in continuity condition  after traveling wave transformation:
\begin{equation}\label{euler}
\begin{cases}
\rho_i(u_i-c)u_{ix}	+\rho_iv_iu_{iy}=-P_{ix},\\
\rho_i(u_i-c)v_{ix}+\rho_iv_iv_{iy}=-P_{iy}-g\rho_i,\\
(u_i-c)\rho_{ix}+v_i\rho_{iy}=0,\\
u_{ix}+v_{iy}=0,
\end{cases}
\end{equation}
 where $g$ represents the gravitational acceleration, $c$ represents the travelling wave speed and $P_i(x, y)$ represent the pressure, for $i=1,2$.
 And the   boundary conditions are give by
 \begin{equation}
 \begin{cases}
 	v_2=(u_2-c)\eta(x),\quad &\rm{on}\quad y=\eta(x),\\
 	v_1=(u_1-c)\widetilde{\eta}(x),\quad &\rm{on}\quad y=\widetilde{\eta}(x),\\
 	v_2=(u_2-c)\widetilde{\eta}(x),\quad &\rm{on} \quad y=\widetilde{\eta}(x),\\
 	v_1=0,\quad &\rm{on}\quad y=-d.
 \end{cases}
 \end{equation}
 We assume that the pressure is continuous, which implies
 \[
 \begin{split}
 P_2=P_{atm},\quad&\rm{on}\quad y=\eta(x),\\
 P_1=P_2,\quad&\rm{on}\quad y=\widetilde{\eta}(x),
 \end{split}
 \]
 where $P_{atm}$ represents the  atmospheric pressure.\\
 We also make following assumptions to ensure that no stagnation point exists neither in $\Omega_1$ and $\Omega_2$:
 \[
 u_i-c<0,\quad \rm {in}\quad \Omega_i.
 \]
 \subsection{Equivalent formulations}
 Inspired by \cite{walsh1}, we introduce the stream function $\Psi_i(x,y)$ by
 \[
 \Psi_{iy}(x,y)=-\sqrt{\rho_i} v,\quad \Psi_{ix}(x,y)=\sqrt{\rho_i}(u-c).
 \]
And the existence of $\Psi_i $ is ensure by the  $4$-th equation of \eqref{euler}.\\
It is not difficult to find that
$\Psi_i$ are constants  at the boundary. Moreover, we standardize the boundary values of $\Psi_i$ by
\[
\begin{cases}
\Psi_1=\Psi_2=0,\quad&\rm{on}\quad y=\widetilde\eta(x),\\
\Psi_1=-p_1,\quad &\rm{on} \quad y=-d,\\
\Psi_2=-p_2,\quad &\rm{ on}\quad y=\eta(x),
\end{cases}
\]
 where $p_1=\int_{-d}^{\widetilde{\eta}(x)}\sqrt{\rho_1}(u_1(x,y)-c)dy$ and $p_2=-\int_{\widetilde{\eta}(x)}^{\eta(x)}\sqrt{\rho_2}(u_2(x,y)-c)dy$ are constants. The level sets of $\Psi_i$ will be referred to as the streamlines of the
 flow.\\
 Moreover, from  Bernoulli’s law, we define the energy function as
 \[
 E_i=P_i+\rho_i\frac{(u_i-c)^2+v_i^2}{2}+\rho_ig(y+d).
 \]
 Note that $E_i$ remain constant along streamlines in $\Omega_i$, since $(u-c)E_{ix}+vE_{iy}=0.$\\
 Similarly, we  obtain that $\rho_i$ are also constants along streamlines in $\Omega_i.$ Therefore, we rewrite them as $\rho_i=\rho_i(-\Psi_i(x,y)).$ And we assume $\rho'_i(-\Psi_i)\leq0$, which is physically reasonable.
 Therefore, we transform the first two equations of \eqref{euler} into
 \begin{equation}\label{jiaocha}
 	\begin{cases}
 		\Psi_{iy}\Psi_{ixy}-\Psi_{ix}\Psi_{iyy}=-P_{ix},\\
 		\Psi_{ix}\Psi_{ixy}-\Psi_{iy}\Psi_{ixx}=-P_{iy}-g\rho_i(-\Psi_i).
 	\end{cases}
 \end{equation}
 Moreover, according to \eqref{jiaocha}, we obtain
 \[
 \nabla ^\perp \Psi_i \cdot\nabla(\Delta\Psi_i-gy\rho_i'(-\Psi_i))=0,
 \]
 which means that $\Delta\Psi_i-gy\rho_i'(-\Psi_i)$ and $\Psi_i$ are orthogonal. In other words, there exist $\beta_1:[-p_1,0]\to\mathbb R$ and $\beta_2:[-p_2,0]\to\mathbb R$, s.t.
 \[
 \Delta\Psi_i-gy\rho_i'(-\Psi_i)=-\beta_i(\Psi_i).
 \]
  And we assume $\beta_i$  decrease as depth increases, that is \[
 \beta_i'<0.
 \]
 Finally, we transform \eqref{euler} and the dynamic boundary conditions above into
 \begin{equation}\label{psi}
 	\begin{cases}
 	&\Delta\Psi_i=gy\rho_i'(-\Psi_i)-\beta_i(\Psi_i)\quad \rm{in}\quad \Omega_i,\quad\rm{for}\quad i=1,2,\\
 	&\frac{|\nabla\Psi_2|}{2}+g\rho_2(-\Psi_2)(y+d)=Q_2\quad \rm{on}\quad y=\eta(x),\\
 	&\frac{|\Psi_1|^2-|\Psi_2|^2}{2}+g(\rho_1(-\Psi_1)-\rho_2(-\Psi_2))(y+d)=Q_1\quad \rm{on}\quad y=\widetilde\eta(x)\\
 	&\Psi_1=\Psi_2=0 \quad \rm{on}\quad y=\widetilde{\eta}(x),\\
 	&\Psi_1=-p_1\quad\rm{on}\quad y=-d,\\
 	&\Psi_2=-p_2\quad\rm{on}\quad y=\eta(x),
 	\end{cases}
 	\end{equation}
 	where $Q_1$ and $Q_2$ are physical parameters.

 	\section{First variational formulation}
 	
 	Following the ideas developed in Chu et al.\cite{2layer}, Chu and Escher \cite{c-e},   Constantin et al. \cite{cons1}, we aim to determine a variational formulation
 	for problem \eqref{psi}.
 	We now show that solutions of \eqref{psi} are characterised as critical points of a certain
 	functional.
 	
 	Let $\Psi_{1p}$, $\Psi_{2p}$, $\widetilde\eta_p$, $\eta_p$ be the  disturbance functions of $(\Psi_1,\Psi_2,\widetilde{\eta},\eta)$.
 	Set $\mathbb G=\{G:\mathbb R\to\mathbb R\  is\ a\ C^2\  function,\ G'' \ vanishes\ nowhere\}$. Then we obtain that $G$ is a concave or convex function. And for any $F_i\in G$, we define
 	\begin{equation}\label{F_i}
 	(\partial_2 F_i(y,\cdot))^{-1}(p)=gy\rho_i'(p)-\beta_i(-p),\quad p\in \mathbb R.
 	\end{equation}
 	For the sake of notational clarity later on, we will denote by $\partial_1 F_i$ and $\partial_2 F_i$ the derivatives of $F_i$ with the respect to the first and second variable, respectively; similarly for
 	higher order derivatives. In fact, from this definition it follows that $F_i$ is a $ C^2$-function.
 	Furthermore, it should be pointed out that $F_i$ is only defined up to a function of $y$, a
 	fact that we are going to exploit in the sequel.
 	In the following Theorem, we restrict the perturbations
 	$(\Psi_{1p},\Psi_{2p},\eta_p,\widetilde\eta_p)$ of $(\Psi_1,\Psi_2,\eta,\widetilde\eta)$  to the subspace
 	\[
 	\mathbb{P}=\left\{(\Psi_{1p},\Psi_{2p},\widetilde{\eta_p},\eta_p)\in\mathbb  G: \int_B\Psi_{1py}dx=0 \ \rm{and}\ \int_S \frac{\partial  \Psi_{2p}}{\partial n_2}dl=0\right\},
 	\]
 	and prove that $(\Psi_1,\Psi_2,\eta,\widetilde\eta)$ is a critical point of $H$ if and only if $(\Psi_1,\Psi_2,\eta,\widetilde{\eta})$ solves \eqref{psi}, where  $B$ represents the bottom of $\Omega_1$ and $ S$ correspond to the upper surface boundary of $\Omega_2$.
 	\begin{theorem}
 	 Let  $F_i$ be defined as
 		in \eqref{F_i}. Then there exist constants $p_1, p_2 \in \mathbb R$ such that,  $(\Psi_1,\Psi_2, \eta,\widetilde\eta)$ is a critical point of
 		the functional
 		\[
 		\begin{split}
 		H(\Psi_1,\Psi_2,\eta,\widetilde\eta)=\iint_{\Omega_1}[\frac{|\nabla\Psi_1|^2}{2}+g\rho_1(p_1)(y+d)-(Q_1+Q_2)-F_1(y,\Delta\Psi_1)]dydx\\
 		+\iint_{\Omega_2}[\frac{|\nabla\Psi_2|^2}{2}+g\rho_2(0)(y+d)-Q_2-F_2(y,\Delta\Psi_2)]dydx,
 		\end{split}
 		\]
 			 if and only if
 			$(\Psi_1,\Psi_2,\eta,\widetilde{\eta})$ solves \eqref{psi}.		 	
 	\end{theorem}
 \begin{proof}
 	Note that
 	\[
 	\delta H=\lim_{\epsilon\to0}\frac{H(\Psi_1+\epsilon\Psi_{1p},\Psi_2+\epsilon\Psi_{2p},\widetilde\eta+\epsilon\widetilde{\eta}_p,\eta+\epsilon\eta_p)-H(\Psi_1,\Psi_2,\widetilde\eta,\eta)}{\epsilon},
 	\]
 	 we  decompose $\delta H$ into four parts: \begin{equation}\label{H}
 		\delta H=\delta H_1+\delta H_2+\delta H_3+\delta H_4.
 	\end{equation}
 	By direct computation, we obtain that
 	\[
 	\begin{split}
 	\delta H_1=-\iint_{\Omega_1}[\Psi_1+\partial_2 F_1(y,\Delta \Psi_1)]\Delta\Psi_{1p}dydx\\
 	-\iint_{\Omega_2}[\Psi_2+\partial_2 F_2(y,\Delta \Psi_2)]\Delta\Psi_{2p}dydx,
 	\end{split}
 	\]
 	\[
 	\delta H_2=\int_S[\frac{|\nabla\Psi_2|^2}{2}+g\rho_2(p_1)(y+d)-Q_2-F_2(y,\Delta \Psi_2)]\eta_pdx,
 	\]
 	\[
 	\delta H_3=\int_S\Psi_2\cdot\frac{\partial\Psi_2}{\partial n_2}dl-\int_{\widetilde{S}}\Psi_2\frac{\partial\Psi_{2p}}{\partial n_2}dl+\int_{\widetilde{S}}\Psi_1\frac{\partial\Psi_{1p}}{\partial n_1}dl-\int_B \Psi_1\Psi_{1py}dx,
 	\]
 	\[
 	\delta H_4=\int_{\widetilde{S}}[\frac{|\nabla\Psi_1|^2-|\nabla\Psi_2|^2}{2}+g(\rho_1(0)-\rho_2(0))(y+d)-F_1(y,\Delta\Psi_1)+F_2(y,\Delta\Psi_2)-Q_1]\widetilde{\eta}_pdx,
 	\]
 	where $\widetilde S$ corresponds to the upper surface boundary of $\Omega_1$ and $n_1$, $n_2$ correspond to the outer normal vectors of $\widetilde S$ and $ S$ respectively.\\
 	According to $(\partial_2 F_i(y,\cdot))^{-1}=gy\rho_i'(p)-\beta_i(-p)$ and the first equation of \eqref{psi}, we obtain
 \[	-\Psi_i=\partial_2 F_i(y,\Delta \Psi_i)\quad\rm{for}\quad i=1,2.\]
 	Moreover, $\delta H_1=0$.\\
 	Combining  the last three equations of \eqref{psi}  with
 	$\int_B\Psi_{1py}dx=0, \  \int_S\frac{\partial\Psi_{2p}}{\partial n_2}dl=0$,
 we obtain
 \[
 \int_{\widetilde{S}}\Psi_1\frac{\partial\Psi_{1p}}{\partial n_1}dl=0
 \]
 and
 \[
 \int_{\widetilde{S}}\Psi_2\frac{\partial\Psi_{2p}}{\partial n_1}dl=0.
 \]
 Therefore, we have
 \[
-\int_B \Psi_1\Psi_{1py}dx+\int_S\Psi_2\frac{\partial\Psi_{2p}}{\partial n_2}dl=p_1\int_B\Psi_{1py}dx-p_2\int_S\frac{\partial\Psi_{2p}}{\partial n_2}dl=0,
 \]
 which implies $\delta H_3=0.$\\
 According to the second equation of \eqref{psi},
 we have
 \[
 \delta H _2=\int_{\widetilde{S}}-F_2(y,\Delta\Psi_2)\eta_pdx=0,
 \]
 by letting $F_2(gy\rho_2'(-\Psi_2)-\beta_2(\Psi_2))=0$ (this is not contradictory to the definition of $F_2$).

 Similarly, we let $F_1(y,\Delta\Psi_1)=F_2(y,\Delta\Psi_2)$.
  Then, we obtain
  \[
  \frac{|\nabla\Psi_1|^2-|\nabla\Psi_2|^2}{2}+g(\rho_1(0)-\rho_2(0))(y+d)=Q_1\quad \rm{on} \quad y=\widetilde{\eta}(x).
  \]
  Therefore, $\delta H_4=0.$

  Now we prove that if $\delta H=0$, then $(\Psi_1,\Psi_2,\eta,\widetilde\eta)$ solves \eqref{psi}. \\
  (i) Take $\eta_p=\widetilde{\eta}_p=0$
 and let $\Psi_{1p}$, $\Psi_{2p}$ satisfy that
 \[
 \begin{cases}
 	\Psi_{1py}=0\quad\rm{on}\quad B,\\
 	\frac{\partial\Psi_{1p}}{\partial n_1}=0 \quad \rm{on}\quad S.
 \end{cases}
 \]
 and
 \[
 \begin{cases}
 	\frac{\partial\Psi_{2p}}{\partial n_2}=0\quad \rm{on } \quad S,\\
 	\frac{\partial\Psi_{2p}}{\partial n_1}=0\quad \rm{on}\quad \widetilde{S}.
 \end{cases}
 \]
 Then \eqref{H}  is simplified as
 \[
 \iint_{\Omega_1}(\Psi_1+\partial_2F_1(y,\Delta\Psi_1))\omega_{1p}dydx+\iint_{\Omega_2}(\Psi_2+\partial_2F_2(y,\Delta\Psi_2))\omega_{2p}dydx=0,
 \]
 where $\omega_{1p}=\Delta\Psi_{1p}$, $\omega_{2p}=\Delta\Psi_{2p}$ are arbitrary functions satisfying
 \[
 \iint_{\Omega_{i}}\omega_{ip}dydx=0, \quad \rm{for}\quad i=1,2.
 \]
 Therefore, there exist $k_1$, $k_2 \in \mathbb R$ satisfying
 \begin{equation}\label{3.4}
 \partial_2 F_i(y,\Delta\Psi_i)=-\Psi_i+k_i, \quad \rm{for}\quad i=1,2.
 \end{equation}
According to the definition of $F_i$, we obtain that
\[
\Delta(\Psi_i-k_i)=gy\rho_i'(-\Psi_i+k_i)-\beta_i(\Psi_i-k_i),
\]
which meets the first equation of \eqref{psi} (with $\Psi_i$ replaced by $\Psi_i-k_i$).
 Moreover, we rewrite the expression of $\delta H$ as
 \begin{equation}\label{new H}
 	\begin{split}
 		\delta H =&-\iint_{\Omega_1}k_1\Delta\Psi_{1p}dydx-\iint_{\Omega_2}k_2\Delta\Psi_{2p}dydx\\
 		&+\int_S[\frac{|\Psi_2|^2}{2}+g\rho_2(0)(y+d)-Q_2-F_2(y,\Delta_2\Psi_2)]\eta_pdx\\
 		&+\int_S\Psi_2\frac{\partial\Psi_2p}{\partial n_2}dl-\int_{\widetilde{S}}\Psi_2\frac{\partial\Psi_{2p}}{\partial n_1}dl+\int_{\widetilde{S}}\Psi_1\frac{\partial\Psi_{1p}}{\partial n_1}dl-\int_B\Psi_1\Psi_{1py}dx\\
 		&+\int_{\widetilde{S}}[\frac{|\Psi_1|^2-|\Psi_2|^2}{2}+g(\rho_1(0)-\rho_2(0))(y+d)-F_1(y,\Delta\Psi_1)+F_2(y,\Delta\Psi_2)-Q_1]\widetilde{\eta}_pdx=0.
 	\end{split}
 \end{equation}

 (ii)
 Take $\eta_p=\widetilde{\eta}_p=0$ and let
 \[
 \begin{cases}
 \Delta\Psi_{1p}=0 \quad in\quad \Omega_1,\\
 \Psi_{1py}=0\quad \rm{on}\quad B,\\
 \frac{\partial\Psi_{1p}}{\partial n_1}=f_1\quad \rm{on}\quad \widetilde{S},
 \end{cases}
 \]
 and
 \[
 \begin{cases}
 	\Delta\Psi_{2p}=0\quad\rm{in} \quad \Omega_2,\\
 	\frac{\partial\Psi_{2p}}{\partial n_1}=f_2\quad\rm{on}\quad\widetilde{S},\\
 	\frac{\partial\Psi_{2p}}{\partial n_2}=0\quad \rm{on}\quad S,
 \end{cases}
 \]
 hold for $\Psi_{1p}$, $\Psi_{2p}$,
 where $f_1$ is an arbitrary smooth functions on $B$ and $f_2$ is an arbitrary smooth function on $S$ with
 \[
 \int_Bf_1dx=0,\quad \int_Sf_2dl=0.
 \]
 Then \eqref{new H} is transformed into
 \[
 \int_{\widetilde{S}}\Psi_2f_2dl-\int_{\widetilde{S}}\Psi_1f_1dx=0,
 \]
 which implies $\Psi_1=k_1$, $\Psi_2=k_2$ on $\widetilde{S}$
 with $\Psi_i$ replaced by $\Psi_i-k_i$.

 (iii) Take $\eta_p=\widetilde{\eta}_p=0$, $\Psi_{1p}, \Psi_{2p}$ be the solution of following equations:
 \[
 \begin{cases}
 \Delta\Psi_{1p}=0\quad \rm{in}\quad \Omega_1,\\
 \frac{\partial\Psi_{1p}}{\partial n_1}=0\quad \rm{on}\quad \widetilde{S},\\
 \Psi_{1yp}=f_1\quad \rm{on}\quad B,
 \end{cases}
 \]
 and
 \[
 \begin{cases}
 	\Delta\Psi_{2p}=0\quad\rm{in}\quad \Omega_2,\\
 	\frac{\partial\Psi_{2p}}{\partial n_1}=0\quad \rm{on}\quad\widetilde{S},\\
 	\frac{\partial\Psi_{2p}}{\partial n_2}=f_2\quad \rm{on}\quad S,
 \end{cases}
 \]
 where  $f_1$ is an arbitrary smooth function on $B$ and $f_2$ is an arbitrary smooth function on $S$ with
 \[
 \int_Bf_1dx=0,\quad\int_Sf_2dl=0.
 \]
 Therefore, \eqref{new H} is transformed into $\int_S\Psi_2f_2dl-\int_B\Psi_1f_1dx=0$, which implies
 \begin{equation}\label{3.6}
 \begin{cases}
 	\Psi_1-k_1=-k_B\quad\rm{on}\quad B,\\
 	\Psi_2-k_2=-k_S\quad\rm{on}\quad S.
 	\end{cases}
  \end{equation}
  Taking $k_B=-p_1$ and $k_S=-p_2$, we obtain the last two equations of \eqref{psi} with $\Psi_i$ replaced by $\Psi_i-k_i$.

(iv) Let $\Psi_{1p}=\Psi_{2p}=0$  and $\eta_p$, $\widetilde{\eta}_p$ be the arbitrary smooth functions.
Therefore \eqref{new H} is transformed into
\[
\begin{split}
&\int_S(\frac{|\nabla\Psi_2|^2}{2}+g\rho_2(p_1)(y+d)-Q_2-F_2(y,\Delta\Psi_2))\eta_pdx\\
&+\int_{\widetilde{S}}(\frac{|\nabla\Psi_1|^2-|\nabla\Psi_2|^2}{2}+g(\rho_1(0)-\rho_2(0))(y+d)+F_2(y,\Delta\Psi_2)-F_1(y,\Delta\Psi_1)-Q_1)\widetilde{\eta}_pdx=0.
\end{split}
\]
Moreover, we have
\[
\begin{cases}
	\frac{|\nabla\Psi_2|^2}{2}+g\rho_2(p_1)(y+d)-Q_2-F_2(y,\Delta\Psi_2)=0,\\
	\frac{|\nabla\Psi_1|^2-|\nabla\Psi_2|^2}{2}+g(\rho_1(0)-\rho_2(0))(y+d)+F_2(y,\Delta\Psi_2)-F_1(y,\Delta\Psi_1)-Q_1=0.
\end{cases}
\]
According to \eqref{3.4} and \eqref{3.6}, we obtain \[\partial_2F_2(y,\Delta\Psi_2)=-k_S\quad \rm {on}\quad S,\]
or
\[
\partial_2F_2(\eta(x),g\eta(x)\rho'(-k_S)-\beta(k_S))=-k_S.
\]
From the definition of $F_2$, it is only related to $y$.
We may therefore require
\[
F_2(y,gy\rho_2'(-k_S)-\beta(k_S))=0,\quad y\geq-d.
\]
Then $F_2(\eta(x),\Delta\Psi_2(x,\eta(x))=F_2(\eta(x),g\rho'(-k_S)-\beta(k_S))=0$, for $-\pi<x<\pi$, which leads to the second equation of \eqref{psi}.
Similarly, we let $F_1(y,\Delta\Psi_1)=F_2(y,\Delta\Psi_2)$ on $y=\widetilde{\eta}(x)$ to obtain the third equation of \eqref{psi}.
 \end{proof}

\section{Second variation}
 Beginning with a critical point $(\Psi_1,\Psi_2,\eta,\widetilde{\eta})\in \mathbb G$. A pair of variations of the critical point is denoted by $(\Psi_{1p},\Psi_{2p},\eta_{p},\widetilde\eta_p)\in \mathbb  G$ and $(\overline\Psi_{1p},\overline\Psi_{2p},{\overline\eta}_{p},{\overline{\widetilde\eta}}_p)\in\mathbb G$.
We also let
$\omega_{i}=\Delta\Psi_i$, $\omega_{ip}=\Delta\Psi_{ip}$ and $\overline\omega_{ip}=\Delta\overline\Psi_{ip}$, for $i=1,2$.  	Then  the second variation of $H$ is calculated in the follow theorem.
	
\begin{theorem}
	If $(\Psi_1,\Psi_2,\eta,\widetilde\eta)$ is the solution of \eqref{psi}, then $\delta^2H$ is expressed as
	\begin{equation}\label{delta2H}
	\begin{split}
	\delta^2H=	&\iint_{\Omega_1}[\nabla\overline{\Psi}_{1p}\cdot\nabla\Psi_{1p}-\partial_{22}F_1(y,\omega_1)\omega_{1p}\overline{\omega}_{1p}]dydx+\iint_{\Omega_2}[\nabla\overline{\Psi}_{2p}\cdot\nabla\Psi_{2p}-\partial_{22}F_2(y,\omega_2)\omega_{2p}\overline{\omega}_{2p}]dydx\\
		&+\int_{\overline{S}}\overline\Psi_{2p}\frac{\partial\Psi_{2p}}{\partial n_1}dl+\int_S\Psi_{2y}\frac{\overline\Psi_{2p}}{\partial n_2}\eta_pdl+\int_S\Psi_{2y}\overline{\eta}_p\frac{\partial \Psi_{2p}}{\partial n_2}dl\\
		&-\int_{\widetilde{S}}(\overline\Psi_{2p}+\Psi_{2y}\overline{\widetilde{\eta}}_p)\frac{\partial\Psi_{2p}}{\partial n_2}dl+\int_{\widetilde{S}}\Psi_{1y}\overline{\widetilde{\eta}}_p\frac{\partial \Psi_{1p}}{\partial_{n_1}}dl+\int_{\widetilde{S}}\Psi_{1y}\frac{\partial\overline\Psi_{1p}}{\partial{n_1}}\widetilde{\eta}_pdl-\int_{\widetilde{S}}\Psi_{2y}\frac{\partial\Psi_{2p}}{\partial n_1}\widetilde{\eta}_pdl
		\\&+\int_S\{g\rho_2(p_1)+[\frac{1}{2}|\nabla\Psi|^2]_y\}\eta_p\overline{\eta}_pdx-\int_S\partial_2F_2(y,\omega_2)\overline\omega_{2p}\eta_pdx\\
		&-\int_S[\partial_1F_2(y,\omega_2)+\partial_2F_2(y,\omega_2)\omega_{2y}]\eta_p\overline\eta_pdx
		+\int_{\widetilde{S}}[-\partial_2 F_{1}(y,\omega_1)\overline{\omega}_{1p}+\partial_2 F_{2}(y,\omega_2)\overline\omega_{2p}]\widetilde{\eta}_pdx\\
		&+\int_{\widetilde{S}}[g(\rho_1(0)-\rho_2(0))-\partial_1F_1(y,\omega_1)-\partial_2F(y,\omega_1)\omega_{1y}+\partial_1F_2(y,\omega_2)+\partial_2F_2(y,\omega_2)\omega_{2y}]\widetilde{\eta}\overline{\widetilde{\eta}}_pdx.
	\end{split}
	\end{equation}
\end{theorem}	
	\begin{proof}
Noting that $\Psi_i=-\partial_2F_i(y,\omega_1)$, we obtain $\delta^2 H=\delta^2H_1+\delta^2 H_2+\delta^2 H^3+\delta ^2 H_4$, where
	\[
	\begin{split}
	\delta^2 H_1=&-\iint_{\Omega_1}[\overline\Psi_{1p}+\partial_{22}F_1(y,\omega_1)\overline\omega_{1p}]\omega_{1p}dydx
	-\int_{\widetilde{S}}\overline{\eta}_p[\Psi_1+\partial_2F_1(y,\omega_1)]\omega_{1p}dx\\
	&+\int_{\widetilde{S}}[\Psi_2+\partial_2F_2(y,\omega_2)]\omega_{1p}\overline{\widetilde{\eta}}_pdx-\int_S[\Psi_2+\partial_2F_2(y,\omega_2)]\omega_{2p}\overline{\eta}_pdx\\
	&-\iint_{\Omega_2}[\overline\Psi_{2p}+\partial_{22}F_2(y,\omega_2)\overline\omega_{2p}]\omega_{2p}dydx\\
	=&-\iint_{\Omega_1}[\overline\Psi_{1p}+\partial_{22}F_1(y,\omega_1)\overline\omega_{1p}]\omega_{1p}dydx-\iint_{\Omega_2}[\overline\Psi_{2p}+\partial_{22}F_2(y,\omega_2)\overline\omega_{2p}]\omega_{2p}dydx\\
	=&\iint_{\Omega_1}[\nabla\overline{\Psi}_{1p}\cdot\nabla\Psi_{1p}-\partial_{22}F_1(y,\omega_1)\omega_{1p}\overline{\omega}_{1p}]dydx+\iint_{\Omega_2}[\nabla\overline{\Psi}_{2p}\cdot\nabla\Psi_{2p}-\partial_{22}F_2(y,\omega_2)\omega_{2p}\overline{\omega}_{2p}]dydx\\
	&-\int_{\widetilde{S}}\overline{\Psi}_{1p}\frac{\partial\Psi_{1p}}{\partial n_1}dl+\int_{\widetilde{S}}\overline\Psi_{2p}\frac{\partial\Psi_{2p}}{\partial n_1}dl-\int_S\overline\Psi_{2p}\frac{\partial\Psi_{2p}}{\partial n_2}dl+\int_B\overline\Psi_{1p}\frac{\partial\Psi_{1p}}{\partial y}dx.
	\end{split}
	\]
Noting that $\delta\int_Sfdl=\int_S\delta fdl+\int_Sf_y\delta\eta dl$ and
$\delta\int_{\widetilde{S}}fdl=\int_{\widetilde{S}}\delta fdl+\int_{\widetilde{S}}f_y\delta\widetilde\eta dl$ (see \cite{cons2}), $\delta^2 H_2$ is given by
	\[
	\begin{split}
	\delta^2H_2=&\int_S[\Psi_{2x}\overline\Psi_{2px}+\Psi_{2y}\Psi_{2py}-\partial_2 F_{2}(y,\omega_2)\overline{\omega}_{2p}]\eta_pdx\\
	&+\int_S[\Psi_{2x}\Psi_{2xy}
	+\Psi_{2y}\Psi_{2yy}+g\rho_2(p_1)-\partial_1 F_{2}(y,\omega_2)-\partial_2F_2(y,\omega_2)\omega_{2y}]\eta_p\overline{\eta}_pdx\\
	=&\int_S\Psi_{2y}\frac{\partial\overline\Psi_{2p}}{\partial n_2}\eta_pdl+\int_S\{g\rho_2(p_1)+[\frac{1}{2}|\nabla\Psi_2|^2]_y\}\eta_p\overline{\eta}_pdx\\
	&-\int_S\partial_2F_2(y,\omega_2)\overline\omega_{2p}\eta_pdx-\int_S[\partial_1F_2(y,\omega_2)+\partial_2F_2(y,\omega_2)\omega_{2y}]\eta_p\overline\eta_pdx.
	\end{split}
	\]
	Similarly, the remaining two terms is calculated as follows
	\[
	\begin{split}
		\delta^2H_3=&\int_S(\overline\Psi_{2p}+\Psi_{2y}\overline{\eta}_p)\frac{\partial\Psi_{2p}}{\partial n_2}dl-\int_{\widetilde{S}}(\overline\Psi_{2p}+\Psi_{2y}\overline{\widetilde{\eta}}_p)\frac{\partial{\Psi_{2p}}}{\partial n_2}dl\\
		&+\int_{\widetilde{S}}(\overline\Psi_{1p}+\Psi_{1y}\overline{\widetilde{\eta}}_p)\frac{\partial{\Psi_{1p}}}{\partial n_1}dl-\int_B\overline{\Psi}_{1p}\Psi_{1py}dx,
	\end{split}
	\]
	\[
 \begin{split}
 	\delta^2H_4&=\int_{\widetilde{S}}[\Psi_{1x}\overline{\Psi}_{1px}+\Psi_{1y}\overline\Psi_{1py}-\partial_2 F_{1}(y,\omega_1)\overline{\omega}_{1p}+\partial_2 F_{2}(y,\omega_2)\overline\omega_{2p}]\widetilde{\eta}_pdx\\
 	&+\int_{\widetilde{S}}[g(\rho_1(0)-\rho_2(0))-\partial_1F_1(y,\omega_1)-\partial_2F(y,\omega_1)\omega_{1y}+\partial_1F_2(y,\omega_2)+\partial_2F_2(y,\omega_2)\omega_{2y}]\widetilde{\eta}\overline{\widetilde{\eta}}_pdx\\
 	=&\int_{\widetilde{S}}\Psi_{1y}\frac{\partial\overline{\Psi}_{1p}}{\partial n_1}\widetilde{\eta}_pdl-\int_{\widetilde{S}}\Psi_{2y}\frac{\partial\Psi_{2p}}{\partial n_1}\widetilde{\eta}_pdl+\int_{\widetilde{S}}[-\partial_2 F_{1}(y,\omega_1)\overline{\omega}_{1p}+\partial_2 F_{2}(y,\omega_2)\overline\omega_{2p}]\widetilde{\eta}_pdx\\
 	&+\int_{\widetilde{S}}[g(\rho_1(0)-\rho_2(0))-\partial_1F_1(y,\omega_1)-\partial_2F(y,\omega_1)\omega_{1y}+\partial_1F_2(y,\omega_2)+\partial_2F_2(y,\omega_2)\omega_{2y}]\widetilde{\eta}\overline{\widetilde{\eta}}_pdx.
 \end{split}
	\]
	\end{proof}
	\begin{definition}
	The traveling wave $(\Psi_1, \Psi_2, \widetilde{\eta}, \eta)$ is linear stable if for any $(\Psi_{1p}, \Psi_{2p}, \widetilde{\eta}_p, \eta_p)\in\mathbb G$, the quadratic form $\delta^2 H$ is nonnegative.
	\end{definition}
	
	Now we give the quadratic form of $\delta^2 H$ by taking $(\Psi_{1p},\Psi_{2p},\eta_p,\widetilde{\eta}_p)=(\overline{\Psi}_{1p},\overline{\Psi}_{2p},\overline\eta_p,\overline{\widetilde{\eta}}_p)$ in \eqref{delta2H}.
	\begin{equation}\label{quadratic}
		\begin{split}
	\delta^2H=		&\iint_{\Omega_1}[\nabla{\Psi}_{1p}\cdot\nabla\Psi_{1p}-\partial_{22}F_1(y,\omega_1)|\omega_{1p}|^2]dydx+\iint_{\Omega_2}[\nabla{\Psi}_{2p}\cdot\nabla\Psi_{2p}-\partial_{22}F_2(y,\omega_2)|\omega_{2p}|^2]dydx\\
			&+\int_{\overline{S}}\Psi_{2p}\frac{\partial\Psi_{2p}}{\partial n_1}dl+\int_S\Psi_{2y}\frac{\Psi_{2p}}{\partial n_2}\eta_pdl+\int_S\Psi_{2y}{\eta}_p\frac{\partial \Psi_{2p}}{\partial n_2}dl\\
			&-\int_{\widetilde{S}}(\Psi_{2p}+\Psi_{2y}{\widetilde{\eta}}_p)\frac{\partial\Psi_{2p}}{\partial n_2}dl+\int_{\widetilde{S}}\Psi_{1y}{\widetilde{\eta}}_p\frac{\partial \Psi_{1p}}{\partial_{n_1}}dl+\int_{\widetilde{S}}\Psi_{1y}\frac{\partial\Psi_{1p}}{\partial{n_1}}\widetilde{\eta}_pdl-\int_{\widetilde{S}}\Psi_{2y}\frac{\partial\Psi_{2p}}{\partial n_1}\widetilde{\eta}_pdl
			\\&+\int_S\{g\rho_2(p_1)+[\frac{1}{2}|\nabla\Psi|^2]_y\}\eta_p^2dx-\int_S\partial_2F_2(y,\omega_2)\omega_{2p}\eta_pdx\\
			&-\int_S[\partial_1F_2(y,\omega_2)+\partial_2F_2(y,\omega_2)\omega_{2y}]\eta_p^2dx
			+\int_{\widetilde{S}}[-\partial_2 F_{1}(y,\omega_1){\omega}_{1p}+\partial_2 F_{2}(y,\omega_2)\omega_{2p}]\widetilde{\eta}_pdx\\
			&+\int_{\widetilde{S}}[g(\rho_1(0)-\rho_2(0))-\partial_1F_1(y,\omega_1)-\partial_2F(y,\omega_1)\omega_{1y}+\partial_1F_2(y,\omega_2)+\partial_2F_2(y,\omega_2)\omega_{2y}]\widetilde{\eta}{\widetilde{\eta}}_pdx.
		\end{split}
	\end{equation}

	Therefore, suitable conditions to ensure \eqref{delta2H} is nonnegative are necessary for us to obtain the linear stable results.
	%when we try to obtain the stability results, we need to find suitable conditions under which

	\begin{theorem}
	If  the surface $\eta$ and the interface $\widetilde\eta$ are unperturbed and $\partial_{22}F_i(y,\omega_i)<0$ hold for $i=1,2,$,  then a classical travelling wave of \eqref{psi} is linearly stable.
		
	\end{theorem}
	We  provide the above theorem, which ensure that the traveling wave is linear stable. Since the proof of the theorem is simple, we omit it here. For the study of more sufficiency conditions, we will leave it as the future topic.

\section*{Acknowledgement}
\hskip\parindent
\small
We declare that the authors are ranked in alphabetic order of their names and all of them have the same contributions to this paper.

\end{document}